\documentclass[preprint,12pt,fleqn]{elsarticle}

\usepackage{amssymb,amsfonts,amsmath,amscd,amsthm,latexsym}
\usepackage{graphicx}
\usepackage{mathptmx}      % use Times fonts if available on your TeX system
\usepackage[english]{babel}
\usepackage{geometry}
\usepackage[all]{xy}
\usepackage{appendix}
\usepackage[active]{srcltx}
\usepackage{color}

\newtheorem{thm}{Theorem}[section]
\newtheorem{cor}[thm]{Corollary}

\newtheorem{prop}[thm]{Proposition}
\theoremstyle{definition}
\newtheorem{defn}[thm]{Definition}

\newtheorem{nim}[thm]{}
\newtheorem{rem}[thm]{Remark}

\newcommand{\Hom}{\mathrm{Hom}}

\newcommand{\End}{\operatorname{End}}
\newcommand{\Ind}{\operatorname{Ind}}
\newcommand{\Int}{\operatorname{Int}}
\newcommand{\Aut}{\operatorname{Aut}}
\newcommand{\Out}{\operatorname{Out}}

\newcommand{\Ker}{\operatorname{Ker}}

\newcommand{\h}{\operatorname{hU}}

\newcommand{\Br}{\operatorname{Br}}

\newcommand{\Od}{\mathcal{O}}

\begin{document}

\begin{frontmatter}



\title{Fusions and Clifford extensions}

\author{Tiberiu Cocone\c t}
\ead{tiberiu.coconet@math.ubbcluj.ro}
\address{Faculty of Economics and Business Administration, Babe\c s-Bolyai University, Str. Teodor Mihali, nr.58-60 , 400591 Cluj-Napoca, Romania}
\author{Andrei Marcus}
\ead{andrei.marcus@math.ubbcluj.ro}
\address{Faculty of Mathematics, Babe\c s-Bolyai University, Cluj-Napoca, Romania}
\author{Constantin-Cosmin Todea\corref{cor1}}
\ead{Constantin.Todea@math.utcluj.ro}
\address{Department of Mathematics, Technical University of Cluj-Napoca, Str. G. Baritiu 25,  Cluj-Napoca 400027, Romania}

\cortext[cor1]{Corresponding author}

\begin{abstract}  We introduce $\bar G$-fusions of local pointed groups on a block extension $A=b\mathcal{O}G$, where $H$ is a normal subgroup of the finite group $G$, $\bar G=G/H$, and $b$ is a $G$-invariant block of $\mathcal{O}H$. We show that certain Clifford extensions associated to these pointed groups  are invariant under group graded basic Morita equivalences.
\end{abstract}


\end{frontmatter}
% \linenumbers

\section{Introduction}\label{sec1}

Pointed groups, organized as a category, have been introduced by Llu\'\i s Puig \cite{Puig2} and \cite{Puig2m}  as a useful tool in modular representation theory.  Puig also associated to a local pointed group on a block algebra certain $k^*$-groups, or twisted group algebras.

The aim of this paper is to extend these constructions to the case of blocks of normal subgroups, by identifying pointed groups with (isomorphism classes of) certain  indecomposable group graded bimodules (in a way similar to \cite[Appendix]{L}) or indecomposable projective modules, and obtain the $k^*$-groups as Clifford extensions associated to these objects.  Note that we will only deal with automorphisms of $p$-groups, and fusions will be regarded as  stabilizers of modules under a suitable action. In the end, we prove that these Clifford extensions are preserved by the group graded basic Morita equivalences introduced in \cite{CM}.

Let us briefly present our setting, which is the same as in \cite{CM} and \cite{CM2}. For any unexplained notions and results we refer to \cite{The}, \cite{Pu} and \cite{Ma}. Let $\mathcal{O}$ be a complete discrete valuation ring with residue field $k$  of  characteristic $p>0$, $k$ is  not assumed to be algebraically closed. Let $H$ be a normal subgroup of a finite group $G$, let $\bar{G}:=G/H$,  and we denote by $\bar{g}=gH$ an element in $\bar{G}$, where $g\in G$.

Let $b$ be a $G$-invariant block of $\mathcal{O}H$ (which is as a $p$-permutation  $G$-algebra), so $b$ remains a primitive idempotent in the $G$-fixed subalgebra $(\Od H)^G$. Then $A:=\Od Gb$ is a $\bar{G}$-graded $\mathcal{O}G$-interior  algebra (that is, the structural map $\Od G\rightarrow A$ is a homomorphism of $\bar{G}$-graded algebras) with identity component $B:=\Od Hb$. For any subgroup $L$ of $G$, the group algebra $\Od L$ is regarded as a $\bar{G}$-graded subalgebra of $\Od G$ in an obvious way, through the isomorphism $\bar{L}=LH/H \simeq L/L \cap H$.

Let  $P_{\gamma}\le G_{\{b\}}$ be a local pointed group on $B$ and $i\in \gamma$ be a primitive idempotent of $B^P$. Then $B^Pi$ is an indecomposable projective $B^P$-module, while $Ai$ is an indecomposable $\bar G$-graded $(A,\Od P)$-bimodule. Note that $iAi$ is $\bar{G}$-graded subalgebra of $A$, with endomorphism algebra
$\End_A(Ai)^{\mathrm{op}}\simeq iAi$
as $\bar G$-graded ${\Od P}$-interior algebras. One gets a Clifford theoretical context by viewing $P$ as a normal subgroup in its normalizer $N_G(P)$ or in the semidirect product $P\rtimes\Aut(P)$. Our main objects of study are the stabilizers of the isomorphism classes of  $B^Pi$, respectively $Ai$, and the associated residual Clifford extensions.

We should mention that block extensions are the focus of \cite{KuPu} without considering group gradings, but introducing so called $S$-permutation groups instead (see \cite[1.6 and 2.4]{KuPu}). Fusion systems on $B$ are discussed in \cite[Part IV]{AKO}, in \cite[Section 3]{KS}, in \cite[Section 8]{Pu} and in \cite{YYZ}, but here we bring in the analysis of the grading.

The paper is organized as follows. In Section 2 we recall the needed facts on group gradings on modules and bimodules, and on Clifford extensions. In Section 3 we introduce the group $\Aut^{\bar G}(P)$ of $\bar G$-automorphisms of $P$, and its subgroup $F_A^{\bar G}(P_\gamma)$ of $(A,\bar G)$-fusions of $P_\gamma$ (where $A$ is, more generally, a $\bar G$-graded $G$-interior algebra), which is defined as the stabilizer of the $G$-graded $(A,\mathcal{O}P)$-bimodule $Ai$ under a natural action of $\Aut^{\bar G}(P)$. When $\bar G$ is trivial, our definition is equivalent to \cite[Definition 2.5]{Puig2m}.

In Section 4 we discuss automorphisms of $P$ determined by $G$-conjugation, leading to the group $E_G^{\bar G}(P_\gamma)$ of $(G,\bar G)$-fusions of $P_\gamma$. We show that in the case of block extensions, the group of $(A,\bar G)$-fusions is isomorphic to the group of $(G,\bar G)$-fusions.

Next, we consider residual Clifford extensions of indecomposable modules associated with $P_\gamma$, and we generalize to the case of block extensions the results of \cite[Section 6]{Puig2m}. There are two distinct constructions, as mentioned above, and we show in Section 5 that they are isomorphic. The next two sections are devoted to the study of the behavior of these Clifford extensions  under several important operations: the extended Brauer construction of \cite{CM2}, embeddings of $P$-algebras, and tensor products of $\bar G$-graded algebras.

We state here the main result of the paper, which, as we explain in {\ref{n:8.1}}, is a generalization of \cite[Theorem 6.4]{CM} and of the main result of \cite{PuZh}, and is concerned with local equivalences induced by group graded basic Morita equivalences. First, we recall the recently introduced definition of $\bar G$-graded basic Morita equivalence.
\begin{defn} \label{defn11} \cite[Theorem 3.10, Definition 4.2.]{CM}
Let $G$ and $G'$ be finite groups, let $\omega:G\rightarrow \bar{G}$ and $\omega':G'\rightarrow\bar{G}$ be group epimorphisms with $H=\Ker \omega$ and $H'=\Ker \omega'$. Let $b$ be a $G$-invariant block of $\Od H$, and $b'$ a $G'$-invariant block of $\Od H'$. Consider the canonical group epimorphism
\[\phi:G\times G'\to \bar{G}\times \bar{G}\] and let $\Delta:=\Delta(\bar{G}\times \bar{G})$ denote the diagonal subgroup of $\bar{G}\times \bar{G}.$   Set $\phi^{-1}(\Delta)=:K\leq G\times G'.$

Let $P_{\gamma}$ be a defect pointed group of $G_{\{b\}}$ and $P'_{\gamma'}$ be a defect pointed group of $G'_{\{b'\}}.$ Let $M$ be an indecomposable $\Od(H\times H')$-module associated with $b\otimes (b')^{\mathrm{op}}$ (that is, $M$ is an $(\Od H,\Od H')$-bimodule such that $bMb'=M$) providing a Morita equivalence between the block algebras $B:=\Od Hb$ and $B':=\Od H'b'.$ If $M$ extends to a $\Od K$-module, then we say that the $\bar{G}$-graded algebras $A:=\Od Gb$ and $A':=\Od G'b'$ are \textit{$\bar{G}$-graded Morita equivalent}.

In this situation, by \cite[Theorem 3.10]{CM},  $M$ has a vertex $\ddot{P}\leq K$, and the projections $G\times G'\to G$ and $G\times G'\to G'$ restrict to the projections $\ddot{P}\to P$ and $\ddot{P}\to P'.$ If, in addition, one of these two last projections is an isomorphism, then
\[\tag{1}\label{groupiso}P\simeq P',\]
and we say that there is a \textit{basic $\bar{G}$-graded Morita equivalence} between the block extensions $A$ and $A'$.
\end{defn}

The main result of this paper is Theorem \ref{thm6.1} below, which will be proved in Section \ref{sec8}.  We assume that the $\Od K$-module $M$ induces a basic $\bar G$-graded Morita equivalence between $A$ and $A'$, as in Definition \ref{defn11}. Let $\ddot{P}\leq K$ be a  vertex of $M$, and let $\ddot{N}$ be an $\Od \ddot{P}$-source module of $M.$ Consider the $\Od \ddot{P}$-interior algebra
\[S:=\End_{\Od}(\ddot{N}).\]
Then our assumptions, together with the isomorphism (\ref{groupiso}), determine the $P\simeq \ddot{P}$-interior algebra embedding
\[\tag{2}\label{idembbedding}f:(\Od G)_{\gamma}\to S\otimes_{\mathcal{O}}(\Od G')_{\gamma'}.\]

Let $Q$ be a subgroup of $ P$ and let $Q'$ be the subgroup of $P'$ corresponding to $Q$ via the isomorphism (\ref{groupiso}). Since $S$ is a Dade $p$-algebra, and $Q\simeq Q'\simeq \ddot{Q}\leq \ddot{P}$, by applying the Brauer construction to the embedding  (\ref{idembbedding}), we obtain the $N_P(Q)$-algebra embedding
\[\tag{3}\label{idquotientdembbedding}\bar{f}:(\Od G)_{\gamma}(Q)\to S(Q)\otimes_{k }(\Od G')_{\gamma'}(Q').\] If $Q_{\delta}$ is a local pointed group included in $P_{\gamma}$, then the embedding (\ref{idquotientdembbedding}) provides an unique local point $\delta'$ corresponding to $\delta$ such that   $Q'_{\delta'}$ corresponds to $Q_{\delta}$ and $Q'_{\delta'}\leq P_{\gamma'}.$ In general, the above construction yields a bijective correspondence between all the local points included in $P_{\gamma}$ and those included in $P'_{\gamma'}$ such that $Q\simeq Q'$.  We denote (see Section \ref{s:sec4})
\[E:=E_{G}^{\bar{G}}(Q_{\delta})= N_G(Q_{\delta})/C_H(Q), \qquad E':=E_{G'}^{\bar G}(Q'_{\delta'})=N_{G'}(Q'_{\delta'})/C_{H'}(Q').\]

\begin{thm} \label{thm6.1} With the above notations,  if $A$ is basic $\bar G$-graded Morita equivalent to $A'$ then:

{\rm 1)}  There is a group isomorphism
\[F_A^{{\bar G}}(Q_{\delta})\simeq F_{A'}^{{\bar G}}(Q'_{\delta'}); \]

{\rm 2)}  There is an $E\simeq E'$-graded basic Morita equivalence
\[kN_G(Q_{\delta})b_{\delta}\sim kN_{G'}(Q'_{\delta'})b'_{\delta'}.\]
\end{thm}

Finally, the last result of this paper is  Corollary \ref{c:final}, which says that the residual Clifford extensions (or $k^*$-groups in Puig's terminology) associated to local pointed groups, are preserved by $\bar G$-graded basic Morita equivalences.  This is a consequence of Theorem \ref{thm6.1} and of our treatment of these Clifford extensions  as group graded endomorphism algebras in Sections \ref{sec5} and \ref{sec6}.

%%%%%%%%%%%%%%%%%%%%%%%%%%%%%%%%%%%%%%%%%%%%%%%%%%%%%%%%%%%%%%%%%%%%%%%%%%%%%%%%%%%%%%%%%%%%%
\section{Preliminaries}\label{sec2}

In this section we recall some results from  Clifford theory in the language of group graded algebra, as introduced by E. C. Dade in \cite{Dade} and \cite{Dade2}. As in the introduction, let $G$ be a finite group, $H$ a normal subgroup of $G$, and let $\bar G=G/H$. It is often more flexible to consider a group epimorphism $\omega:G\to \bar G$ such that $H=\Ker \omega$. Let $A=\bigoplus_{g\in \bar G}A_g$ be a $\bar G$-graded $\mathcal{O}$-algebra.

\begin{nim} \label{s:diagonal} Let $A'$ be another $\bar G$-graded $\mathcal{O}$-algebra. Then the diagonal subalgebra of $A\otimes A'$ is the $\bar G$-graded algebra
\[\Delta(A\otimes A')=\bigoplus_{g\in\bar G}(A_g\otimes {A'}_g).\]
Let $M=\bigoplus_{x\in \bar G}M_x$ be a $\bar G$-graded $(A,A')$-bimodule. Then $M_1$ is a $\Delta:=\Delta(A\otimes {A'}^{\mathrm{op}})$-module, and if both $A$ and $A'$ are strongly graded, then we have a natural isomorphism
\[M\simeq (A\otimes {A'}^{\mathrm{op}})\otimes_{\Delta}M_1\]
of $\bar G$-graded $(A,A')$-bimodules.
\end{nim}

\begin{nim} Let again $M$ be a $\bar G$-graded $(A,A')$-bimodule. For $y\in \bar G$ we denote
\[M(y):=\bigoplus_{x\in \bar G}M(y)_x, \qquad \textrm{ where } \ M(y)_x=M_{xy}.\]
Consider also the $\bar G$-graded algebra $A^y$, where
\[A^{y}:=\bigoplus_{g\in \bar G}(A^y)_g, \qquad (A^y)_g=A_{y^{-1}gy}.\]
Then it is clear that $M(y)$ is a $\bar G$-graded $(A,{A'}^{y})$-bimodule.

Let $\varphi$ be the automorphism of $A'$ given by conjugation with an invertible element $u\in A'_y$ of degree $y\in\bar G$, that is, $\varphi(a')={}^{u}a=ua'u^{-1}$. Then $\varphi:{A'}^y\to {A'}$    is an isomorphism of  $\bar G$-graded algebras.

Moreover, if we define $M_{\varphi}$ as left $A$-module, with the right $A'$-module structure given by $m\cdot a'=m\varphi(a')$, then $(M_{\varphi})(y)=M(y)_{\varphi}$ becomes a $\bar G$-graded $(A,A')$-bimodule.
\end{nim}

\begin{nim}\label{remgeneralGgraded} Assume that $A$ is a crossed product between $B:=A_1$ and $\bar G$. This means that the group $\mathrm{hU}(A)=\bigcup_{g \in\bar G}(A^\times\cap A_g)$ of homogeneous units of $A$ is a group extension of $B^\times$ by $\bar G$.

Let $g\in \bar G$ and $u_g\in A^{\times}\cap A_g$, so $A_g=Bu_g=u_gB.$ Let $V$ be an $A_1$-module and define ${}^gV:=V$ as a set, with the $B$-module structure given by
\[a_1\cdot v=(u_g^{-1}a_1u_{g})v\]
for any $v\in V$ and $a_1\in B$. It is easy to verify the existence of an $B$-module isomorphism
\[A_g\otimes_{B}V\simeq {}^gV.\]
This $A_1$-module is called the $g$-conjugate of $V$, and $V$ is called $\bar G$-invariant, if $V\simeq {}^g V$ as $B$-modules for all $g\in \bar G$.
\end{nim}

\begin{nim} \label{Clifford-ext} Let $A$ and $V$ be as before, and consider the $\bar G$-graded $A$-module $M=A\otimes_{B}V$. Then the endomorphism algebra $A':=\End_A(M)^{\mathrm{op}}$ is a $\bar G$-graded algebra with  homogeneous $g$-component for $g\in \bar G$ given by
\[A'_g=\Hom_{A\textrm{-Gr}}(M,M(g))\simeq \Hom_{A_1}(M_1,M_g)\simeq \Hom_{A_1}(M_x,M_{xg})\]
for any $g,x\in \bar G$. In this way, $M$ becomes a $\bar G$-graded $(A,A')$-bimodule. Note that $A'$ is a crossed product if and only if $V$ is a $\bar G$-invariant $B$-module. Then the graded Jacobson radical $J_{\mathrm{gr}}(A')$ equals $J(B)A=AJ(B)$, and $\overline{A'}:=A/J_{\mathrm{gr}}(A')$ is still a crossed product of $B/J(B)$ and $\bar G$. In this case, the group extension $\mathrm{hU}(A')$ is called the {\it Clifford extension} of $V$, while $\mathrm{hU}(\overline{A'})$ is the {\it residual Clifford extension} of $V$.
\end{nim}

\begin{nim} \label{projcover} Assume, in addition, that $V$ is a $G$-invariant simple $B$-module,  and let $P_V\rightarrow V\rightarrow 0$ be a projective cover of $V$. Then $P_V$ is a $G$-invariant $B$-module, and $\mathrm{ann}_B(V)$ is a $G$-invariant ideal of $B$. Moreover, we have the isomorphisms
\begin{align*}\End_{A}(A\otimes_{B}P_V)/J_{gr}(\End_A(A\otimes_{B}P_V))  &\simeq \End_A(A\otimes_{B}V)   \\
    &\simeq \End_{A/\mathrm{ann}_B(V)A}(A/\mathrm{ann}_B(V)A\otimes_{B/\mathrm{ann}_B(V)}V)\end{align*}
of $\bar G$-graded algebras, so in particular, the Clifford extension of $V$ is isomorphic to the residual Clifford extension of $P$.
\end{nim}

\begin{nim} \label{s:skew*} We will often use the construction of a $\bar G$-graded crossed product (generalized skew group algebra) from \cite[Section 9]{Pu}. Assume that $B$ is an $H$-interior $G$-algebra, that is, we have a map $\mathcal{O}H\to B$ of $G$-algebras. Let $A=B\otimes_{\mathcal{O}H}\mathcal{O}G$, and define the multiplication by
\[(a\otimes g)(b\otimes h)=a\cdot{}^gb\otimes gh.\]
We get that $A$ is a $\bar G$-graded $G$-interior  algebra, which will be denoted $A=B*\bar G$. In fact, any $\bar G$-graded $G$-interior  algebra is obtained in this way.

As a variation, assume that $K$ is another normal subgroup of $G$ such that $B$ is $K$-{\it trivial}, by which we mean that $K$ acts trivially on  $B$, and the restriction  $K\cap H\to B$ is also trivial. Then $B$ becomes a $KH/K$-interior $G/K$-algebra, so denoting $\tilde G=G/KH$,  we may construct the $\tilde G$-graded $G$-interior algebra $\tilde A=B*\tilde G$.

We may compare the Clifford extensions of the simple $B$-module $V$ which occur in this setting. If
\[I:=\{g\in G\mid {}^gV\simeq V \textrm{ as } B\textrm{-modules}\}\]
denotes the stabilizer of $V$ in $G$, then clearly, the stabilizer of $V$ in $\bar G$ is $\bar I=I/H$, while the stabilizer of $V$ in $\tilde G$ is $\tilde I=I/KH$. Moreover, since $V$ extends trivially to $A_{\bar K}=B*\bar K$, where $\bar G=KH/H$, it is not difficult to see that there is a map of $\tilde I$-graded algebras from $\End_A(A\otimes_BV)$ to $\End_{\tilde A}(\tilde A\otimes_BV)$, so the Clifford extension of $V$ with respect to $A$ is obtained from the Clifford extension of $V$ with respect to $\tilde A$ by inflation from $\tilde I$ to $\bar I$.
\end{nim}

%%%%%%%%%%%%%%%%%%%%%%%%%%%%%%%%%%%%%%%%%%%%%%%%%%%%%%%%%%%%%%%%%%%%%%%%%%%%%%%%%%%%%%%%%%%%%
\section{Group graded $A$-fusion groups} %\label{sec3}
In this section  $A$ is a $\bar{G}$-graded $G$-interior algebra, with identity component $B:=A_1$, which is an $H$-interior $G$-algebra.

\begin{nim} \label{hyp3.1} Let $\gamma$ be a point of $B^P$, where $P$ is any subgroup of $G$, and let $i\in\gamma$ be a primitive idempotent. We will regard the pointed group $P_\gamma$ as the isomorphism class of the $\bar G$-graded $(A,\mathcal{O}P)$-bimodule summand $Ai$ of $A$, which is essentially the point of view in \cite[Appendix]{L}.  Then its endomorphism algebra $\End_A(Ai)^{\mathrm{op}}\simeq iAi$ is $\bar{G}$-graded, $P$-interior subalgebra of $A$. We will assume that the structural map $P\rightarrow (iAi)^{\times}$ is injective.

Our Definition \ref{defnGfus} below of fusions is more general than that of \cite[Section 8]{Pu}, and it is equivalent, when $\bar G$ is trivial, to \cite[Definition 2.5]{PuLo}, but for simplicity, we will only consider automorphisms of $P$, and not fusions between two pointed groups. Moreover, it is also an improvement of \cite[Definition 6.4]{CM}, where still a smaller automorphism group of $P$ is considered.
\end{nim}

\begin{defn}\label{defnGaut}

 a) The \textit{group of $\bar G$-automorphisms of $P$} is
\[\Aut^{\bar{G}}(P):=\{(\varphi,\bar{g})\mid \varphi \in \Aut(P),\ g\in G \text{ such that } \overline{\varphi(u)}={}^{\bar{g}}\bar{u},\ \forall u\in P\},\]

b) The \textit{group of interior $\bar G$-automorphisms of $P$} is
\[\Int^{ \bar{G}}(P):=\{(c_v,\bar{g})\mid c_v\in \Int(P),\ v\in P,\ \bar{g}\in\bar{G} \text { such that } \overline{c_v(u)}={}^{\bar{g}}\overline{u},\ \forall u\in P\}\]
\end{defn}

Clearly, $\Aut^{\bar{G}}(P)$ is a subgroup of $\Aut(P)\times \bar{G}$. In particular, if $(\varphi,\bar{g})\in\Aut^{\bar{G}}(P)$ then $\bar{g}\in N_{\bar{G}}(\overline{P})$.
Moreover, $\Int^{ \bar{G}}(P)$ is a normal subgroup of $\Aut^{\bar{G}}(P)$, and there is a group homomorphism
\[P\rightarrow \Int^{\bar{G}}(P), \qquad v\mapsto (c_v,\overline{v}).\] We denote $\Aut^{\bar{G}}(P)/\Int^{\bar{G}}(P)$ by
$\Out^{\bar{G}}(P)$, which will be called the group of\textit{ exterior $\bar G$-automorphisms} of $P$.

\begin{defn}\label{defnGfus}

a) The \textit{group of $(A,\bar{G})$-fusions} of $P_\gamma$ is the subgroup
\[F_A^{\bar{G}}(P_{\gamma}):=\{(\varphi,\bar{g})\in\Aut^{\bar{G}}(P) \mid Ai\simeq (Ai)(\bar{g}^{-1})_{\varphi} \text{ as } \bar{G}\text{-graded } (A, \Od P)\text{-bimodules}\}.\]
of $\Aut^{\bar{G}}(P)$.

b) The \textit{group of exterior $(A,\bar{G})$-fusions} of $P_{\gamma}$ is the subgroup
\[\tilde{F}_A^{\bar{G}}(P_{\gamma}):= F_A^{\bar{G}}(P_{\gamma})/ \Int^{\bar{G}}(P),\]
of $\Out^{\bar{G}}(P)$.
\end{defn}

\begin{rem} There is an action of $\Aut^{\bar{G}}(P)$ on the set of isomorphism classes of $\bar{G}$-graded $(A,\Od P)$-bimodules given by
\[(\varphi, \bar{g})\cdot M:=(M(\bar{g}^{-1}))_{\varphi},\]
where  $(\varphi, \bar{g})\in\Aut^{\bar{G}}(P)$ and $M$ is a representative of an isomorphism class of  $\bar{G}$-graded $(A,\Od P)$-bimodules. With these notations we obtain \[F_A^{\bar{G}}(P_{\gamma})=\operatorname{Stab}_{\Aut^{\bar{G}}(P)}(Ai).\]
\end{rem}

The next result is a graded variant of \cite[Corollary 2.13]{PuLo}.

\begin{prop}\label{propnormfus} Assume that the structural group homomorphism $P\rightarrow (iAi)^{\times}$ is injective. There is a  homomorphism
\[\Phi:N_{\h(iAi)}(Pi)\rightarrow F_A^{\bar{G}}(P_{\gamma}),\]
of finite groups which induces the isomorphisms
\[\overline{\Phi}:N_{\h(iAi)}(Pi)/C_{(iBi)^{\times}}(Pi)\rightarrow F_A^{\bar{G}}(P_{\gamma})\]
and
\[\frac{ N_{\h(iAi)}(Pi)}{PC_{(iBi)^{\times}}(Pi)}\rightarrow \tilde{F}_A^{\bar{G}}(P_{\gamma}).\]
\end{prop}

\begin{proof} It is easy to verify that $F_A^{\bar{G}}(P_{\gamma})$ is a subgroup of $\Aut^{\bar{G}}(P)$. Let $\Phi$ be defined by
\[\Phi(a)=(\varphi_a, \bar{g})\]
for any $a\in (iAi)^{\times}\cap (iAi)_{\bar{g}}$ normalizing $Pi$, where $\varphi_a$ is the automorphism of $P$ given by the conjugation with $a$. It is clear that $\Phi$ is a well-defined homomorphism of groups.

First we verify  that $\Phi$ is surjective. For this, let $(\varphi,\bar{g})\in F_A^{\bar{G}}(P_{\gamma})$ such that there is $f:Ai\rightarrow (Ai)(\bar{g}^{-1})_{\varphi}$ an isomorphism of $\bar{G}$-graded $(A,\Od P)$-bimodules.  Since $Ai$ and $(Ai)(\bar{g}^{-1})_{\varphi}$ are  direct summands of $A$ as left $A$-modules  it follows, by the Krull-Schmidt theorem, that there is $\overline{f}\in\Aut_A(A)$ such that $\overline{f}|_{Ai}=f$. Then there is $a\in A^{\times}$ such that $\overline{f}(b)=ba$ for any $b\in A$. In particular
\[f:Ai\rightarrow (Ai)(\bar{g}^{-1})_{\varphi}\quad f(bi)=bia\]
for any $b\in A$. Since $f$ is an isomorphism of $\bar{G}$-graded bimodules, we obtain that $a$ is an homogeneous unit, that is,  $a\in A^{\times}\cap A_{\bar{g}}$ for some $\bar g\in \bar G$. We show that $ai=ia$ by using the well-known isomorphisms
\[\psi :iAi\rightarrow \End_A(Ai), \qquad \psi(ibi)=\alpha_{ibi}\]
and
\[\psi^{-1}: \End_A(Ai)\rightarrow iAi, \qquad \psi^{-1}(\alpha)=\alpha(i),\]
for any $b\in A$ and $\alpha\in  \End_A(Ai)$, where $\alpha_{ibi}$ is the right multiplication by $ibi$. It is clear that $f$ induces and isomorphism
\[\tilde{f}:\End_A(Ai)\rightarrow\End_A(Ai)\quad \tilde{f}(\alpha)=f\circ \alpha \circ f^{-1}\]
for any $\alpha\in\End_A(Ai)$.  Since $\tilde{f}(\alpha_i)=\mathrm{id}_{Ai}$ it follows that
\begin{equation}\label{eq1}
i=ia^{-1}ia.
\end{equation}
 For any $b'\in A$ we apply (\ref{eq1})  and the definitions of $f,f^{-1}, \alpha_i, \alpha_{aia^{-1}}$ to obtain
\[\tilde{f}(\alpha_{aia^{-1}})(b'i)=b'ia^{-1}aia^{-1}a=b'i=\tilde{f}(\alpha_i)(b'i)\]
hence $\alpha_{aia^{-1}}=\alpha_i$. Next we apply $\psi^{-1}$ to get $i=aia^{-1}$. It follows that
\[ai=ia=iai\in(iAi)^{\times}\cap (iAi)_{\bar{g}}\] and
\[\Phi(ai)=(\varphi_{ai},\bar{g})=(\varphi,\bar{g}).\]
Next, we verify that  $C_{(iBi)^{\times}}(Pi)=\Ker \Phi$. If $a\in \Ker \Phi$ then
\[\Phi(a)=(\varphi_a,\bar{g})=(\mathrm{id}_P,\overline{1}),\]
hence $Ai\simeq (Ai)_{\varphi_a}$ as $\bar{G}$-graded $(A,\Od P)$-bimodules and $\varphi_a(u)=u$ for any $u\in P$. The same argument as for surjectivity assure us that $a\in (iBi)^{\times}$. The reverse inclusion is straightforward.
\end{proof}

%%%%%%%%%%%%%%%%%%%%%%%%%%%%%%%%%%%%%%%%%%%%%%%%%%%%%%%%%%%%%%%%%%%%%%%%%%%%%%%%%%
\section{$G$-fusion and  $A$-fusion in the group algebra case} \label{s:sec4}

Let,  as in the Introduction,  $A=\Od G b$ and  $B=\Od Hb$. We regard $B$ as an $H$-interior $G$-algebra and let $P_{\gamma}\leq G_{\{b\}}$ be a local pointed group on $B$, where $P$ is a $p$-subgroup of $G$.

\begin{defn}\label{defnGGfus} a) The \textit{group of $(G,\bar{G})$-fusions} of $P_\gamma$ is
\[E_G^{\bar{G}}(P_{\gamma}):=N_G(P_{\gamma})/C_H(P).\]

b) The \textit{group of exterior $(G,\bar{G})$-fusions} of $P_{\gamma}$ is
\[\tilde{E}_G^{\bar{G}}(P_{\gamma}):=N_G(P_{\gamma})/PC_H(P).\]
\end{defn}

The next result is a graded variant of \cite[Theorem 3.1]{Puig2}.

\begin{prop}\label{propEisoF} With the above notations the following group isomorphisms  hold:
\[{E}_G^{\bar{G}}(P_{\gamma})\simeq N_{\h(iAi)}(Pi)/C_{(iBi)^{\times}}(Pi)\simeq F_A^{\bar{G}}(P_{\gamma}).\]
\end{prop}

\begin{proof} The second isomorphism is true because in this case, $iAi$ satisfies the hypotheses of Proposition \ref{propnormfus} applied to our algebras.
Let $g\in N_G(P_{\gamma})$, then $gig^{-1}\in \gamma$, hence there is $a_1\in (B^P)^{\times}$ such that ${}^gi={}^{a_1}i$. It follows that $(a_1^{-1}g)i=i(a_1^{-1}g)$.  Next, we define
\[\Theta: N_G(P_{\gamma})\rightarrow N_{\h(iAi)}(Pi)/C_{(iBi)^{\times}}(Pi), \qquad g\mapsto \Theta(g):=\overline{ia_1^{-1}g},\]
which is well defined since $ia_1^{-1}g\in N_{\h(iAi)}(Pi)$.

It is easy to verify that $\Theta$ is a homomorphism and that $\Ker  \Theta=C_H(P)$. (Note that this does not require the assumption that $A$ is a block extension; it holds for any $\bar G$-graded $G$-interior algebra $A$ as in \ref{hyp3.1}.)

 To show that $\Theta$ is surjective, consider the elements
\[\overline{a_g}\in N_{\h(iAi)}(Pi)/C_{(iBi)^{\times}}(Pi), \qquad g\in G.\]
By Proposition \ref{propnormfus} we have  $(\varphi_{a_g},\bar{g})\in F_{A}^{\bar{G}}(P_{\gamma})$, and denote by $\varphi$ the map $\varphi_{a_g}$. By the definition of $F_A^{\bar{G}}(P_{\gamma})$ we obtain
\[Ai\simeq (Ai)(\bar{g}^{-1})_{\varphi}\] as $\bar{G}$-graded $(A,\Od P)$-bimodules. It follows
\[iAi\simeq (iAi)(\bar{g}^{-1})_{\varphi}\]
as $\bar{G}$-graded $(\Od P,\Od P)$-bimodules. Moreover, since $\Br_P(i)\neq 0$ and $A$ has a direct summand  isomorphic to $\Od P$, it follows that $iAi$ has a direct summand isomorphic to $\Od P$ as $(\Od P,\Od P)$-bimodules, hence $iAi$ has a direct summand isomorphic to $(\Od P)(\bar{g})_{\varphi^{-1}}$. But
\[iAi=i\Od Gi=\bigoplus_{x\in Y} \Od[PxP],\]
where $Y$ is a subset of $[P\setminus G / P]$ with $[N_G(P_{\gamma})/PC_H(P)]\subset Y$. Since the elements from $Y\setminus [N_G(P_\gamma)/PC_H(P)]$ are not in $N_G(P)$, arguing by  contradiction, we get
\[(\Od P)(\bar{g})_{\varphi^{-1}}\simeq \Od [PxP]\]
for some $x\in Y\cap N_G(P)$. In conclusion there is $x\in Y\cap N_G(P_{\gamma})$ with $\bar{x}=\bar{g}$ such that $\varphi=\varphi_x.$
\end{proof}

%%%%%%%%%%%%%%%%%%%%%%%%%%%%%%%%%%%%%%%%%%%%%%%%%%%%%%%%%%%%%%%%%%%%%%%%%%%%%
\section{An isomorphism of Clifford extensions}\label{sec5}

As in the previous section, let $A=\mathcal{O}Gb$, $B=\mathcal{O}Hb$; let $b$ be a $G$-invariant  block of $B$, let $P_{\gamma}\leq G_{\{b\}}$ be a local pointed group on $B$, and let $i\in\gamma$.
By Puig \cite[Section 6]{Puig2m}, there are two $k^*$-groups (twisted group algebras) associated with $P_\gamma$. We pursue our module theoretic point of view and we associate to $P_\gamma$ two group graded endomorphism algebras, which turn out to be isomorphic in the group algebra case. These crossed product algebras do not yield $k^*$-groups in general, because we do not assume that $k$ is algebraically closed.

\begin{nim} We begin by setting our notation:
\[E:=E_G^{\bar{G}}(P_{\gamma})\simeq N_G(P_{\gamma})/C_H(P), \qquad \tilde{E}:=\tilde{E}_G^{\bar{G}}(P_{\gamma})\simeq N_G(P_{\gamma})/PC_H(P),\]
\[ F:=F_A^{\bar{G}}(P_{\gamma})\simeq N_{\h(iAi)}(Pi)/C_{( iBi)^{\times}}(Pi),\]
\[ \tilde{F}:=\tilde{F}_A^{\bar{G}}(P_{\gamma})\simeq N_{\h(iAi)} (Pi)/PC_{(iBi)^{\times}}(Pi).\]
By Proposition \ref{propEisoF} we know that there exist isomorphisms $E\simeq F$ and $\tilde{E}\simeq\tilde{F}$.
\end{nim}

\begin{nim} Since $B^P$ is a $C_H(P)$-interior $N_G(P_\gamma)$-algebra, we may consider, as in \ref{s:skew*}, the $E$-graded crossed product $B^P * E$. We also denote
\[\mathcal{E}=\mathcal{E}_G^{\bar G}(P_\gamma):=\End_{B^P * E}(B^P * E\otimes_{B^P}B^Pi)^{\mathrm{op}},\]
which is an $E$-graded crossed product, because $B^Pi$ is an $E$-invariant $B^P$-module.
\end{nim}

\begin{nim} Since $\mathcal{O}P$ is an $F$-algebra, we may consider, as in \ref{s:skew*}, the $F$-graded crossed product
\[C:=\mathcal{O}P * F,\]
which we view as an $F$-interior $F$-graded algebra. We obtain the $F$-graded algebra
\[\mathcal{F}=\mathcal{F}_A^{\bar G}(P_\gamma):=\End_{A\textrm{-Gr}\otimes C^{\mathrm{op}}}(Ai\otimes_{\mathcal{O}P}C)^{\mathrm{op}}.\]
(The notation means that an element $f\in \mathcal{F}$ preserves $\bar G$-gradings, that is, it satisfies $f(Ai_{\bar g})\subseteq Ai_{\bar g}$ for all $\bar g \in \bar G$.)
\end{nim}

The next result generalizes  \cite[Proposition 6.12]{Puig2m}, and also \cite[2.5]{YYZ}, where the case $P\subseteq H$ was considered.

\begin{prop}\label{properonsisofrond} There is an isomorphism of $E\simeq F$-graded $\mathcal{O}$-algebras
\[\mathcal{E}\simeq\mathcal{F}.\]
\end{prop}

\begin{proof} We have $\mathcal{E}_1\simeq\mathcal{F}_1$, since
\[\mathcal{E}_1=(\End_{B^P}(B^Pi))^{\mathrm{op}}\simeq iB^Pi\]
and
\[\mathcal{F}_1=(\End_{A\textrm{-Gr}\otimes(\mathcal{O}P)^{\mathrm{op}}}(Ai))^{\mathrm{op}}=(iBi)^P=iB^Pi.\]

Let $\bar{g}\in E$, where $g\in N_G(P_{\gamma})$, and let $\varphi\in F$ be the correspondent of $\bar g$ through the isomorphism $E\simeq F$, so $\varphi$ is the conjugation by $\overline{ia_1^{-1}g}$, according to the proof of Proposition \ref{propEisoF}. We have the isomorphisms
\begin{align*}\mathcal{F}_{\varphi} &=\left(\End_{A\textrm{-Gr}\otimes C^{\mathrm{op}}}(Ai\otimes_{\Od P} C)_{\varphi}\right)^{\mathrm{op}}  \\
                                         &\simeq \Hom_{A\textrm{-Gr}\otimes (\Od P)^{\mathrm{op}}}(Ai,Ai\otimes_{\Od P}\Od P * \varphi) \\
                                         &\simeq \Hom_{A\textrm{-Gr}\otimes (\Od P)^{\mathrm{op}}}(Ai,(Ai)_{\varphi}) \\
                                         &\simeq \Hom_{A\textrm{-Gr}\otimes (\Od P)^{\mathrm{op}}}((Ai)(\bar{g}^{-1})_{\varphi},(Ai)_{\varphi}),
\end{align*}
where the last isomorphism follows from Definition \ref{defnGfus}. We also have
\begin{align*}\mathcal{E}_{\bar g} &=\left(\End_{B^P * E}(B^P * E\otimes_{B^P}B^Pi)\right)^{\mathrm{op}}  \\
                                                      &\simeq\Hom_{B^P}(B^Pi,\bar g^{-1}\otimes B^Pi)   \\
                                                      &\simeq\Hom_{B^P}(B^Pi,{}^{\bar g^{-1}}(B^Pi));
\end{align*}
note that here, as in \ref{remgeneralGgraded},
\[\bar g\otimes B^Pi\simeq{}^{\bar g}(B^Pi),\]
as left $B^P$-modules, with the multiplication  given by
\[b\cdot (b'i):={\ }^{\bar g}b~b'i=gbg^{-1}b'i\] for all $b,b'\in B^P$. Next, we define an $E\simeq F$-graded homomorphism
\[\Psi:\mathcal{E}\rightarrow \mathcal{F}\]
as follows; let $f\in \mathcal{E}$, so
\[f:B^Pi\rightarrow {}^{\bar g}(B^Pi),\quad b\mapsto f(bi)=b\cdot f(i)={\ }^{\bar g}b\ f(i).\]
It is easy to check that
\[\Psi(f):Ai\rightarrow (Ai)_{\varphi},\quad a\mapsto \Psi(f)(a)=af(i)\]
is a homomorphism of $\bar G$-graded left $A$-modules and that $\Psi$ is an $E\simeq F$-graded homomorphism. We verify that
$\Psi(f)$ is a homomorphism of right $\Od P$-modules. Indeed, for any $u\in P$ and $a\in Ai$, we have
\[\Psi(f)(au)=auf(i) \qquad \textrm{and}\qquad \Psi(f)(a)\varphi(u)=af(i)\varphi(u).\]
It is enough to prove that \[uf(i)=f(i)\varphi(u).\]
Here we regard ${}^{\bar g}(Ai)$ as an $(A,\Od P)$-bimodule, with scalar multiplication  given by
\[a\cdot a'u:={\ }^{\bar g}a~a'u,\]
while ${}^{\bar g}(B^Pi)$ is a right $\Od P$-submodule of ${}^{\bar g}(Ai)$. We know that $f(i)$ is an element of
${}^{\bar g}(B^Pi)$, but $f(i)$ can be viewed in ${}^{\bar g}(Ai)$, and here it is fixed by $\Delta P\subseteq A\otimes (\Od P)^{\mathrm{op}}$; thus
\[u\cdot f(i)u^{-1}=f(i),\] and it follows that
\[{}^{\bar g}u \ f(i)=f(i)u.\]
But $\bar g$ corresponds to $\varphi$ in the isomorphism $E\simeq F$, so $g$ is chosen such that we have $\varphi(u)={}^{\bar g}u$;  this concludes the proof.
\end{proof}

\begin{rem} Under the more general assumptions of \ref{hyp3.1}, the map $\Psi$ is an injective homomorphism, inducing an isomorphism of $E$-graded algebras between $\mathcal{E}$ and $\mathcal{F}_E=\bigoplus_{\phi\in E}\mathcal{F}_\phi$.
\end{rem}

\begin{nim} \label{barcalEF} We now pass to the residual Clifford extensions of the indecomposable module $B^Pi$ and of the indecomposable $\bar G$-graded $(A,\mathcal{O}P)$-bimodule $Ai$.  Consider the $E$-graded $k$-algebra
\[\overline{\mathcal{E}}:=\End_{B(P_{\gamma}) * E}(B(P_{\gamma}) * E\otimes V_{\gamma})^{\mathrm{op}},\]
where  $V_{\gamma}$ is the  unique simple module over the simple quotient $B(P_\gamma)$ of $B^P$ associated with the point $\gamma$.
Consider also the $F$-graded $k$-algebra
\[\overline{\mathcal{F}}:=\mathcal{F}/J_{gr}(\mathcal{F}).\]
\end{nim}

\begin{cor} There is an isomorphism of $E\simeq F$-graded $k$-algebras
\[\overline{\mathcal{E}}\simeq\overline{\mathcal{F}}.\]
\end{cor}

\begin{proof} We have the isomorphisms
\[\mathcal{E}_1/J(\mathcal{E}_1)\simeq \End_{B^P}(B^Pi/J(B^Pi))\simeq \End_{B^P}(V_{\gamma})^{\mathrm{op}}\simeq\End_{B(P_{\gamma})}(V_{\gamma}),\]
so by \ref{projcover} we obtain
\[\mathcal{E}/J_{gr}(\mathcal{E})\simeq\overline{\mathcal{E}}.\]
The statement now follows immediately by Proposition \ref{properonsisofrond}.
\end{proof}

\begin{rem} Observe that the $C_H(P)$-interior $N_G(P_\gamma)$ algebra $B(P_\gamma)$ is actually $P$-trivial, so as in \ref{s:skew*}, we may construct an $\tilde E$-graded endomorphism algebra $\widetilde{\mathcal{E}}$, whose inflation is $\overline{\mathcal{E}}$. Similarly, the  $C_{(iBi)^{\times}}(Pi)$-interior $N_{\h(iAi)}(Pi)$-algebra $iB^Pi/J(iB^Pi)$ is $P$-trivial, so we have an $\tilde F$-graded crossed product $\widetilde{\mathcal{F}}$, whose inflation is $\overline{\mathcal{F}}$.

\end{rem}

%%%%%%%%%%%%%%%%%%%%%%%%%%%%%%%%%%%%%%%%%%%%%%%%%%%%%%%%%%%%%%%%%%%%%%%%%%%%%%%%%%%%%%
\section{Local Clifford extensions}\label{sec6}

The employment of the Brauer construction gives yet another Clifford extension. We keep the notations and assumptions of the previous section.

\begin{nim} \label{barcalE(P)} The Brauer map
\[\operatorname{Br}_P\!:B^P\rightarrow B(P)\]
is a homomorphism of $C_H(P)$-interior $N_G(P)$-algebras, and it induces the isomorphism
\[B(P_{\gamma}):=B^P/m_{\gamma}\rightarrow B(P)/m_{\Br_P(\gamma)}\]
of $C_H(P)$-interior $N_G(P_{\gamma})$-algebras. Thus we obtain the commutative diagram
\[\begin{xy} \xymatrix{ B^P \ar[r]\ar[d]&B(P)\ar[d]\\
 B(P_{\gamma})\ar[r]^{\simeq~~~~~~~}   &B(P)/m_{\Br_P(\gamma)}}\end{xy}\]
The point $\Br_P(\gamma)$ belongs to the uniquely determined block $b_\gamma$ of $B(P)$. Note that in our group algebra case we have
\[B(P)b_{\gamma}\simeq kC_H(P)b_{\gamma},\]
and
\[B(P)b_{\gamma} * E\simeq
kN_G(P_{\gamma})b_{\gamma}.\]
Denote by $m^{*}$ the ideal $m_{\Br_P(\gamma)}\cap B(P)b_{\gamma}$, which is an $E$-invariant ideal of $B(P)b_{\gamma}$. The simple module $V_{\gamma}$ can also be viewed as a simple $B(P)b_{\gamma}/m^*$-module through the isomorphism $B(P_{\gamma})\simeq B(P)b_{\gamma}/m^*$.
Consider the $E$-graded $k$-algebra
\[\overline{\mathcal{E}}(P):= \End_{(B(P)b_{\gamma}/m^*) * E}\left((B(P)b_{\gamma}/m^*) * E\otimes_{B(P)b_{\delta}/m^*}V_{\gamma}\right)^{\mathrm{op}}.\]

\end{nim}

\begin{prop}\label{prop5.1} With the above notations, there is an isomorphism
\[\overline{\mathcal{E}}\simeq \overline{\mathcal{E}}(P)\]
of $E$-graded algebras.
\end{prop}

\begin{proof} We clearly have the isomorphisms
\[B(P_{\gamma}) * E\simeq (B(P)b_{\gamma}/m^*) * E\simeq (B(P)b_\gamma * E)/m^* * E\]
of $E$-graded algebras, and it follows that the $E$-graded endomorphism of $k$-algebras associated to $V_{\gamma}$ are isomorphic.
\end{proof}

\section{Embeddings and tensor products}\label{sec7}

Our constructions behave well with respect to embeddings and tensor products of $P$-algebras. We record these properties in this section, leaving the proofs to the reader.

\begin{nim}\label{s:embedding} Let $A$ be as in \ref{hyp3.1}, and let $A'\to A$ be an embedding of $\bar G$-graded $P$-algebras. We may assume that $A'=eAe$, where $e\in B^P$ is an idempotent such that $ei=ie=i$. Then, since the map $(eBe)^\times\to B^\times$ sending $a$ to $a+(1-e)$ is a group homomorphism, the primitive idempotent $i\in \gamma\subset B^P$ determines a point $\gamma'$ of $P$ on $B'=eBe$.

We denote $E'=E_G^{\bar G}(P_{\gamma'})$, $F'=F_A^{\bar G}(P_{\gamma'})$, $\mathcal{E}'=\mathcal{E}_G^{\bar G}(P_{\gamma'})$ and $\mathcal{F}'=\mathcal{F}_A^{\bar G}(P_{\gamma'})$. With these notations, \cite[Proposition 6.15, Proposition 6.18]{Puig2m} generalize as follows.
\end{nim}

\begin{prop}\label{p:embedding} The embedding $A'\to A$  of $\bar G$-graded $P$-algebras induces the group isomorphisms $E\simeq E'$ and $F\simeq F'$ and the commutative diagram
\[\begin{xy}
\xymatrix{ \overline{\mathcal{F}}\ar[d]^{\simeq}    &  \overline{\mathcal{E}} \ar[l]\ar[d]^{\simeq} \ar[r]^{\simeq}  & \overline{\mathcal{E}}(P) \ar[d]^{\simeq} \\
           \overline{\mathcal{F}'}         &   \overline{\mathcal{E}'} \ar[l]    \ar[r]^{\simeq} & \overline{\mathcal{E}'}(P)
}\end{xy}\]
of $F$-graded $k$-algebras.
\end{prop}

\begin{nim} \label{s:tensor} The following setting will also be employed in the next section. Let $\omega:G\to \bar G$ and $\omega':G'\to \bar G$ be group epimorphisms such that $\Ker \omega=H$ and $\Ker \omega'=H'$.  Let \[\ddot{G}=(\omega\times\omega')^{-1}(\Delta(\bar G))=\{(g,g')\in G\times G' \mid \omega(g)=\omega(g')\},\]
so $\mathcal{O}G''$ is the diagonal subalgebra of the $\bar G\times\bar G$-graded algebra $\mathcal{O}G\otimes\mathcal{O}G'$ (see \ref{s:diagonal}).

Let $b\in(\Od H)^G$ and  $b'\in(\Od H')^{G'}$ be $\bar G$-invariant blocks of $\Od H$, respectively $\Od H'$. As in the previous sections, let $A=\Od Gb$ and  $B=\Od Hb$, and we also denote $A':=\Od G'b'$ and  $B':=\Od H'b'$. Let
\[\ddot{\Delta}=(b\otimes b')\mathcal{O}\ddot{G}\]
be the diagonal subalgebra of $A\otimes A'$.
\end{nim}

\begin{nim} \label{s:tensorpt} We assume that the $p$-group $P$ is a common subgroup of $G$ and $G'$ such that $\omega(P)=\omega'(P)$ in $\bar G$. Let $P_\gamma$ and $P_{\gamma'}$ be local pointed groups on $B$ and $B'$, respectively. Let $\gamma''=\gamma\times\gamma'$ be the unique local point of $P$ on $B\otimes B'$ such that $\Br_P(\gamma)\otimes\Br_{P}(\gamma') \subseteq\Br_P(\gamma\times\gamma')$ (see \cite[Proposition 5.6]{Puig3}), and the group graded algebras $\overline{\mathcal{E}'}$, $\overline{\mathcal{F}'}$ and $\overline{\mathcal{E}'}(P)$     as in \ref{barcalEF} and \ref{barcalE(P)}. We do the same for $A''$ and $P_{\gamma''}$.
\end{nim}

The following generalization of \cite[Proposition 5.11]{Puig3} essentially follows from the Hom-Tensor interchange property and \cite[Section 5]{CM2}.

\begin{prop}\label{p:tensor} Let $K:=F\cap F'$. Then $K\subseteq F''$, and there is an there is an isomorphism of $K$-graded algebras between $\overline{\mathcal{F}''}_K$ and the diagonal subalgebra of $\overline{\mathcal{F}}_K\otimes \overline{\mathcal{F}'}_K$, compatible with the isomorphisms given by Proposition \ref{properonsisofrond}  and Proposition \ref{prop5.1}, that is, we have the commutative diagram
\[\begin{xy}
\xymatrix{ \overline{\mathcal{F}''}_K \ar[d]^{\simeq}    &  \overline{\mathcal{E}''}_K \ar[l]_\simeq\ar[d]^{\simeq} \ar[r]^{\simeq}  & \overline{\mathcal{E}''}(P)_K \ar[d]^{\simeq} \\
           \Delta(\overline{\mathcal{F}}_K\otimes\overline{\mathcal{F}'}_K)         &   \Delta(\overline{\mathcal{E}}_K\otimes\overline{\mathcal{E}'}_K) \ar[l]_\simeq    \ar[r]^{\simeq} & \Delta(\overline{\mathcal{E}}(P)_K\otimes\overline{\mathcal{E}'}(P)_K)
}\end{xy}\]
of isomorphisms of $K$-graded $k$-algebras.
\end{prop}

%%%%%%%%%%%%%%%%%%%%%%%%%%%%%%%%%%%%%%%%%%%%%%%%%%%%%%%%%%%%%%%%%%%%%%%%%%%%%%%%
\section{Graded basic Morita equivalences and the invariance of $\mathcal{E}$}\label{sec8}

In this section we consider basic Morita equivalences between block extensions, as discussed in \cite{CM} and Section \ref{sec1}, and we give a proof of the main result of the paper, Theorem \ref{thm6.1}.

\begin{nim} \label{n:8.1} We will use the properties of the extended Brauer quotient $\bar N_B^K(Q)$ for $H$-interior $G$-algebras (see \cite{CT}) and $\bar N_A^K(Q)$ for $G$-graded $H$-interior $G$-algebras (see \cite{CM2}), where, in both cases, $K$ is a suitable subgroup of $\Aut^{\bar G}(Q)$. Recall that, in particular, $\bar N_{\mathcal{O}G}^E(Q)\simeq kN_G(Q_\delta)$ as $E$-graded algebras, with $1$-component $kC_H(Q)$ as $C_H(Q)$-interior $N_G(Q_\delta)$-algebras.
\end{nim}

Our main result here improves \cite[Theorem 6.9]{CM}, because there the grading group is the subgroup $N^K_G (Q_\delta)/C_H(Q)$ (defined in \cite[5.1]{CM}) of $E= N_G(Q_{\delta})/C_H(Q)$.

\begin{proof}(\textbf{Proof of Theorem \ref{thm6.1}})  1) The isomorphism holds by the argument of \cite[1.17]{KuPu} (which can be easily adapted to our definition of $\bar G$-fusions) and by  Proposition \ref{p:embedding}.

2) Note that Proposition \ref{propEisoF} and i) provide the isomorphisms
\[E=N_G(Q_{\delta})/C_H(Q) \simeq F_{A}^{{\bar G}}(Q_{\delta})\simeq F_{A'}^{{\bar G}}(Q'_{\delta'})\simeq N_{G'}(Q'_{\delta'})/C_{H'}(Q')=E'.\] Consider  the block $b_{\delta}$ of $kC_H(Q)$ determined by $\delta.$ One can easily see that $b_{\delta}$ is $N_G(Q_{\delta})$-invariant, hence it lies in $kC_H(Q)^{N_G(Q_{\delta})}$, where it is still a primitive idempotent. The same  holds for the block $b_{\delta'}$ of $kC_{H'}(Q')$ determined by $\delta',$ so $b_{\delta'}$ is a primitive idempotent in $kC_{H'}(Q')^{N_{G'}(Q'_{\delta'})}$.

Let $T$ be   a defect group in  $ N_G(Q_{\delta})$ of  $b_{\delta}.$ By using the epimorphism
\[\Br_Q:(\Od H)^{N_G(Q_{\delta})}_T\to (kC_H(Q))^{N_G(Q_{\delta})}_T\] we determine a unique point $\beta\subseteq (\Od H)^{N_G(Q_{\delta})}$ such that $Q_{\delta}\leq N_G(Q_{\delta})_{\beta}\leq G_{\{b\}}.$ By the Mackey decomposition we have
\[(\Od H)_P^G\subseteq \sum_{x\in [N_G(Q_{\delta})\setminus G/ P]}(\Od H)_{N_G(Q_{\delta})\cap P^x}^{N_G(Q_{\delta})},\] so the defect group $T$ of $\beta$ lies in $P^x$, for some $x\in G.$ Hence, by replacing each $P_{\gamma}$ and $Q_{\delta}$ by a $G$-conjugate such that we still have $Q_{\delta}\leq P_{\gamma},$  we may assume that $T\leq P$, and in fact, more precisely, that $T\leq N_P(Q_{\delta}).$  The local pointed group $P_{\gamma}$  forces the existence of a local point $\mu \subseteq (\Od H)^{N_P(Q_{\delta})}$ with the property
\[N_P(Q_{\delta})_{\mu}\leq P_{\gamma}.\] The inclusion $Q\leq N_P(Q_{\delta})\leq N_G(Q_{\delta})$ shows that we may find, if necessary, some $G$-conjugate of $Q_{\delta }$ satisfying
\[Q_{\delta}\leq N_P(Q_{\delta})_{\mu}\leq P_{\gamma}\]
and
\[Q_{\delta}\leq N_P(Q_{\delta})_{\mu}\leq N_G(Q_{\delta})_{\beta},\] since $\delta$ determines $\beta.$

Now, the defect group $T$ of $\beta$ verifies $T\leq N_P(Q_{\delta}).$ The local point $\mu$ determines a local point $\bar{\mu}$ of $T$ on $\Od H$ with $T_{\bar{\mu}}\leq N_P(Q_{\delta})_{\mu}.$ The maximality of $T_{\bar{\mu}}$ forces the equality $T_{\bar{\mu}}=N_P(Q_{\delta})_{\mu}.$ With the help of the  commutative diagram
\[\begin{xy} \xymatrix{ kC_H(Q)^T \ar[rr]^{\Br^{kC_H(Q)}_T} &&kC_H(T)
\\  \mathcal{O}H^T\ar[u]_{\Br^{\mathcal{O}H}_Q}\ar[urr]_{\Br_T^{\mathcal{O}H}}    } \end{xy}\]
 one can easily check that $T_{\Br_Q(\mu)}$ is a defect pointed group of $N_{G}(Q_{\delta})_{b_{\delta}}.$

Let $T'_{\mu'}$ denote the local pointed group corresponding to $T_{\mu}.$  The embedding (\ref{idquotientdembbedding}) gives $Q'_{\delta'}\leq T'_{\mu'}\leq P'_{\gamma'},$ $Q'_{\delta'}\leq T'_{\mu'}\leq N_{G'}(Q'_{\delta'})_{\beta'},$ where $\beta'$ is the unique point determined by $\delta'.$ Further, since $P\simeq P'$ we have $T'=N_{P'}(Q_{\delta'}).$  We claim that $T'_{\Br_{Q'}(\mu')}$ is a defect pointed group of $N_G(Q'_{\delta'})_{b_{\delta'}}.$ Indeed, if $T'_{\mu'}$ is not maximal then there is a local pointed group $\bar{T}'_{\bar{\mu}'}$ with  $T'_{\mu'}\leq \bar{T}'_{\bar{\mu}'}\leq N_{G'}(Q'_{\delta'})_{\beta'}.$ By repeating the argument in the case of $\beta'=b'\beta'$ instead of $\beta=\beta b$, we  get $\bar{T}'\leq T',$ hence $T'_{\mu'}$ is a defect pointed group of $ N_{G'}(Q'_{\delta'})_{\beta'}.$

Since $b$ is basic Morita equivalent to $b',$ by \cite[Section 4 and Theorem 3.10]{CM} there is a ${\bar G}$-graded $P$-interior algebra embedding
\[\tag{4}\label{gradedembbeding}(\Od G)_{\gamma}\to S\otimes (\Od G')_{\gamma'}\] whose identity component is the embedding (\ref{idembbedding}), and it restricts to
\[((\Od H)_{\gamma})^T\to S^{\ddot{T}}\otimes ((\Od H')_{\gamma'})^{T'}.\] We obtain the commutative diagram
 \[\begin{xy} \xymatrix{ ((\Od H)_{\gamma})^T\ar[d]_{\Br^{(\Od H)_{\gamma}}_T}\ar[r]^{f} &S^{\ddot{T}}\otimes ((\Od H')_{\gamma'})^{T'}\ar[d]^{\Br_T^{S\otimes (\Od H')_{\gamma'}}}\\
  ((\Od H)_{\gamma})(T)\ar[r]^{\bar{f}}&S(\ddot{T})\otimes ((\Od H')_{\gamma'})(T') } \end{xy}.\]
Now the local point $\mu$ of $T$ on $(\Od H)_{\gamma}$ verifies
\[(\bar{f}\circ \Br^{(\Od H)_{\gamma}}_{T})(\mu)=(\Br_T^{S\otimes (\Od H')_{\gamma'}}\circ f)(\mu),\]
which means that
\[\bar{1}\otimes \Br^{(\Od H')_{\gamma'}}_{T'}(\mu ')=\Br_T^{S\otimes (\Od H')_{\gamma'}}(1\otimes \mu')=(\Br_T^{S\otimes (\Od H')_{\gamma'}}\circ f)(\mu),\] which implies that $(1\otimes \mu')f(\mu ')\neq 0.$ The algebra homomorphism
\[((\Od H')_{\gamma'})^{T'}\to (S\otimes (\Od H')_{\gamma'})^{T'},\] sending $a$ to $1\otimes a$ maps the point $\mu'$ to $1\otimes \mu',$ which is a $(S\otimes (\Od H')_{\gamma'})^\times$-conjugacy class of a idempotent that is not primitive in general. We obtain the equalities
\[ (1\otimes \mu')f(\mu ')=f(\mu ')(1\otimes \mu')=f(\mu ')\] in $(S\otimes (\Od H')_{\gamma'})^{T'}.$ All the above facts imply that
\[f(l)(S\otimes (\Od G')_{\gamma'})f(l)=f(l)(S\otimes (\Od G')_{\mu'})f(l),\] for some idempotent $l\in \mu.$
Embeddings (\ref{gradedembbeding}), (\ref{idembbedding}) and the above equality provide the ${\bar G}$-graded $T$-interior algebra embedding
\[\tag{5}\label{theneededembbeding}(\Od G)_{\mu}\to S\otimes (\Od G')_{\mu'}\] given by the  composition
\[(\Od G)_{\mu}\simeq ((\Od G)_{\gamma'})_{\mu}\to f(l)(S\otimes (\Od G')_{\gamma'})f(l)=f(l)(S\otimes (\Od G')_{\mu'})f(l)\to S\otimes (\Od G')_{\mu'}\]
of $T$-interior algebra embeddings.

Let $\widehat{\Br_Q(\mu)}$ be the point of $T$,  determined by $\Br_Q(\mu)$, on the extended Brauer quotient $\bar{N}_{\Od H}^E(Q)$, and similarly, let $\widehat{\Br_{Q'}(\mu')}$ be the point of $T'$ on $\bar{N}_{\Od H'}^{E'}(Q')$ determined by $\Br_{Q'}(\mu'),$ according to \cite[Theorem 3.1]{CT}. At this point, \cite[Corollary 3.7]{PuZh}, \cite[Proposition 2.5]{CT} and \cite[Theorem 4.4]{CM2} give the $E$-graded $T$-interior algebra isomorphisms
\[\tag{6}\label{quotientisomorphism}\bar{N}_{(\Od G)_{\mu}}^E(Q)\simeq (\bar{N}_{\Od G}^E(Q))_{\widehat{\Br_Q(\mu)}}\simeq (kN_G(Q_{\delta}))_{\widehat{\Br_Q(\mu)}},\] and respectively  the $E'$-graded $T'$-interior algebra isomorphisms
\[\tag{7}\label{quotientisomorphism'}\bar{N}_{(\Od G')_{\mu}}^{E'}(Q')\simeq \bar{N}_{\Od G'}^{E'}(Q'))_{\widehat{\Br_{Q'}(\mu')}}\simeq (kN_{G'}(Q'_{\delta'}))_{\widehat{\Br_{Q'}(\mu')}}.\]

Finally, we use (\ref{quotientisomorphism}), (\ref{quotientisomorphism'}), and we apply \cite[Proposition 3.8]{PuZh} and \cite[Theorem 5.1]{CM2} to the embedding (\ref{gradedembbeding}) to obtain the $E\simeq E'$-graded $T$-interior algebra embedding
\[(kN_G(Q_{\delta}))_{\widehat{\Br_Q(\mu)}}\to S(\ddot{Q})\otimes (kN_{G'}(Q'_{\delta'}))_{\widehat{\Br_{Q'}(\mu')}}.\]
Note that here we have that  $S(\ddot Q)=\End_{\Od}(\ddot{N}_{\ddot{Q}})$, where $\ddot{N}_{\ddot{Q}}$ is the unique endo-permutation $k\ddot{T}$-module determined by $\ddot{N}$ (where recall that $\ddot{T}\simeq T\simeq T'$).

Consider, as in \cite[6.8]{CM}, the natural maps $\omega:N_G(Q_\delta)\to E$ and $\omega':N_{G'}(Q'_{\delta'})\to E'$, and the subgroup (see \ref{s:tensor} and \ref{s:tensorpt})
\[N_{\ddot{G}}(\ddot{Q}_{\delta\times \delta'})=\{(g,g')\in N_G(Q_\delta)\times N_{G'}(Q'_{\delta'}) \mid \omega(g)=\omega'(g') \}.\]
Then, by \cite[Theorem 3.10 and 3.11]{CM}, we deduce that there is an  indecomposable $kN_{\ddot{G}}(\ddot{Q}_{\delta\times\delta'})$-summand $Y$ of $\Ind_{\ddot{T}}^{N_{\ddot{G}}(\ddot{Q}_{\delta\times\delta'})} (\ddot{N}_{\ddot{Q}})$ such that the $E$-graded bimodule
\[\ddot{Y}:=\Ind_{N_{\ddot{G}}(\ddot{Q}_{\delta\times\delta'})}^{N_G(Q_{\delta})\times N_{G'}(Q'_{\delta'})}(Y)\] determines a basic $E$-graded Morita equivalence between the block extensions $kN_G(Q_{\delta})b_{\delta}$ and $kN_{G'}(Q'_{\delta'})b_{\delta'}.$
\end{proof}

From Theorem \ref{thm6.1}  we deduce our final result.

\begin{cor} \label{c:final} With the above notations and assumptions, there is an isomorphism
\[\overline{\mathcal{F}}_A^{\bar G}(Q_{\delta})\simeq \overline{\mathcal{F}}_{A'}^{\bar G}(Q'_{\delta'})\]
of $F_A^{\bar G}(Q_{\delta})\simeq F_{A'}^{\bar G}(Q'_{\delta'})$-graded algebras.
\end{cor}

\begin{proof} By Proposition \ref{properonsisofrond} and Proposition \ref{prop5.1} it is enough to see that
\[\overline{\mathcal{E}}(Q)\simeq \overline{\mathcal{E'}}(Q') \]
as $E\simeq E'$-graded algebras. Indeed, the $E$-graded Morita equivalence of Theorem \ref{thm6.1} gives, in particular, a Morita equivalence between the blocks $kC_H(Q)b_\delta$ and $kC_{H'}(Q)b'_{\delta'}$, such that the simple $kC_H(Q)b_\delta$ -module $V_\delta$ corresponds to the simple $kC_{H'}(Q)b'_{\delta'}$-module $V'_{\delta'}$. By \cite[Theorem 5.1.18]{Ma}, the $E$-graded Morita equivalence preserves Clifford extensions, hence that statement follows.
\end{proof}

\begin{rem} In the particular case $\bar G=1$, the isomorphism $\overline{\mathcal{F}}_A(Q_{\delta})\simeq \overline{\mathcal{F}}_{A'}(Q'_{\delta'})$ is stated without proof in \cite[7.6.5]{PuLo}. Note that we are only able to prove this isomorphism by using the local Morita equivalence of \ref{thm6.1}.2), which generalizes \cite[Theorem 1.4]{PuZh}, which in turn, is a generalization of \cite[7.7.4]{PuLo}.

\end{rem}

%%%%%%%%%%%%%%%%%%%%%%%%%%%%%%%%%%%%%%%%%%%%%%%%%%%%%%%%%%%%


\begin{thebibliography}{99}

\bibitem{AKO} M. Aschbacher, R. Kessar, B. Oliver. \textit{ Fusion Systems in Algebra and Topology}. (London Mathematica Series, Cambridge, 2011).
\bibitem{CM} T. Cocone\c t, A. Marcus. Group graded basic Morita equivalences. \textit{Journal of Algebra} {\bf489} (2017), 1-24.
\bibitem{CM2} T. Cocone\c t, A. Marcus.  Remarks on the extended Brauer quotient. \textit{Journal of Algebra} {\bf491} (2017), 78-89.
\bibitem{CT} T. Cocone\c t, C.-C. Todea. The extended Brauer quotient of $N$-interior $G$-algebras. \textit{ J. Algebra} {\bf396} (2013), 10--17.
\bibitem{Dade} E.C. Dade. Group graded rings and modules. \textit{Math. Z.} \textbf{174} (1980), 241--262.
\bibitem{Dade2} E.C. Dade. Extending Group Modules in a Relatively Prime Case. \textit{Math. Z.} {\bf186} (1984) 81--98.
\bibitem{KS}  R. Kessar, R. Stancu. A reduction theorem for fusion systems of blocks. \textit{J. Algebra} {\bf319} (2008) 806--823.
\bibitem{KuPu} B. Kulshammer,  L. Puig. Extensions of nilpotent blocks. \textit{Invent. Math.} \textbf{102} (1990), 17--71.
\bibitem{L} M. Linckelmann. On splendid derived and stable equivalences between blocks and finite groups. \textit{J. Algebra} {\bf242} (2001), 819--843.
\bibitem{Ma} A. Marcus. \textit{Representation Theory of Group Graded algebras}. (Nova Science Publishers, 1999).
\bibitem{Puig2} {L. Puig}. Local fusions in block source algebras. \textit{J. Algebra} {\bf104} (1986), 358--369.
\bibitem{Puig2m} {L. Puig}. Pointed groups and construction of modules. \textit{J. Algebra} {\bf116} (1988), 7--129.
\bibitem{Puig3} {L. Puig}. Nilpotent blocks and their source algebras. \textit{Invent. Math.} {\bf93}  (1988), 77--116.
\bibitem{PuLo} L. Puig. \textit{On the Local Structure of Morita and Rickard Equivalences between Brauer Blocks}. (Progress in Mathematics (Boston, Mass.), Vol. \textbf{178}, Birkh\"{a}user, Basel, 1999).
\bibitem{Pu} L. Puig. \textit{Blocks of finite groups: The Hyperfocal Subalgebra of a Block}. (Springer Verlag Berlin Heidelberg, 2002).
\bibitem{PuZh} L. Puig, Y. Zhou. A local property of basic Morita equivalences. \textit{Math. Z.} \textbf{256} (2007), 551--562.
\bibitem{The} J. Th\'{e}venaz. $G$-\textit{algebras and modular representation theory}. (Oxford Math. Mon. Clarendon Press, Oxford, 1995).
\bibitem{YYZ} Y. Zhou. On the $p'$-extensions of inertial blocks. \textit{Proc. Am. Math. Soc.} {\bf144} (2016), 41--54.
\end{thebibliography}
\end{document}